\documentclass[a4paper,twoside,english]{amsart}
\usepackage[T1]{fontenc}
\usepackage[latin9]{inputenc}
\usepackage{babel}
\usepackage{verbatim}
\usepackage{amsthm}
\usepackage{amsmath}
\usepackage{amssymb}
\usepackage[unicode=true,
 bookmarks=true,bookmarksnumbered=true,bookmarksopen=false,
 breaklinks=false,pdfborder={0 0 1},backref=false,colorlinks=false]
 {hyperref}
\hypersetup{
 pdfauthor={V. Mantova and U. Zannier}}

\makeatletter

\pdfpageheight\paperheight 
\pdfpagewidth\paperwidth

\theoremstyle{plain}
\newtheorem{thm}{\protect\theoremname}[section]
 \newcommand\thmsname{\protect\theoremname}
 \newcommand\nm@thmtype{theorem}
 \theoremstyle{plain}

  \theoremstyle{remark}
  \newtheorem{rem}[thm]{\protect\remarkname}
  \theoremstyle{definition}
  \newtheorem*{example*}{\protect\examplename}
  \theoremstyle{definition}
  
  \theoremstyle{plain}
  \newtheorem{lem}[thm]{\protect\lemmaname}
  \theoremstyle{plain}
  \newtheorem{prop}[thm]{\protect\propositionname}
  \theoremstyle{plain}
  
  \theoremstyle{plain}
  \newtheorem{question}[thm]{\protect\questionname}

\makeatletter
  \theoremstyle{definition}
  \newtheorem{my@rem}[thm]{Remark}
  
\makeatother

\makeatother

  \providecommand{\examplename}{Example}
  \providecommand{\lemmaname}{Lemma}
  \providecommand{\propositionname}{Proposition}
  \providecommand{\remarkname}{Remark}
  \providecommand{\theoremname}{Theorem}
\providecommand{\theoremname}{Theorem}
 \providecommand{\corollaryname}{Corollary}
 \providecommand{\questionname}{Question}

\usepackage[left=3.1cm,right=3.5cm,top=3cm,bottom=3cm]{geometry}

\usepackage{mathtools}

\def\Z{{\Bbb Z}}
\def\G{{\Bbb G}}

\def\P{{\Bbb P}}
\def\C{{\Bbb C}}

\def\B{{\bf B}}

\def\CVD{{\hfill\hfil{\lower 2pt\hbox{\vrule\vbox to 7pt
{\hrule width  5pt\varphifill\hrule}\varphirule}}}\par}

\begin{document} 

\title{On the $p$-adic distribution of torsion values for a section of an abelian scheme}
\author{Brian Lawrence and Umberto Zannier}
\maketitle

\begin{abstract}
Let $A \rightarrow S$ be an abelian scheme over a $p$-adic field, and let $s \colon S \rightarrow A$ be a section.
We study the torsion locus $\bigcup \limits_{n \geq 1} s^{-1}(A[n])$ on $S$,
and we show that torsion points on $S$ of different orders stay away from each other.
\end{abstract}

\section{Introduction}

Let $A \rightarrow S$ be an abelian scheme of relative dimension $g$, and let $s \colon S \rightarrow A$ be a section.  For each positive integer $n$, we may consider $s^{-1}(A[n])$, the maximal closed subscheme of $S$ on which $n s$ agrees with the identity section.  Write 
\[ s^{-1} A_{\mathrm{tors}} = \bigcup \limits_{n \geq 1} s^{-1}(A[n]); \]
we call this set the \emph{torsion locus} of $A \rightarrow S$ relative to the section $s$.

\medskip

Heuristically, one expects that, if $\dim S < g$, then under suitable genericity assumptions the torsion locus is contained in a single $s^{-1}(A[N])$; in the general case,  this is a difficult question in the subject of unlikely intersections. (See for instance  the papers  \cite{CMZ} and \cite{MZ}, of the second author resp.\ with Corvaja and Masser, and  Masser, for finiteness conclusions when $S$ is a curve.)

 If $\dim S \geq g$, one expects, in the absence of special circumstances,  that the torsion locus will be Zariski dense in $S$; and in fact if $S$ is a scheme over $\mathbf{R}$ or $\mathbf{C}$, that it will be dense for the classical topology on $S$ as well.  This issue appears  natural to investigate, and also has shown to admit applications, as we shall briefly recall.
 
 \medskip

Complex density of the torsion locus has been proven in a number of cases. For instance, the Jacobian family over the universal hyperelliptic family of a given genus leads naturally  to a certain section related to the Pell-Abel equation in polynomials, and in this realm complex density of torsion values is known since long ago. (See the second author's survey  paper \cite{Z} for more.) In the same context, but looking at {\it real} density, the problem is subtler; see the recent {\it S\'eminaire Bourbaki} by Serre \cite{S} for a discussion of some applications. The density here was proved by    Bogatyrev \cite{Boga} and, independently,  by the first  author \cite{Law}. 
In  the  general case, Andr\'e, Corvaja and the second author, partly with Gao \cite{ACZG}, studied the torsion locus by means of the ``Betti map'' associated to the section; this is a real-analytic map from the universal cover of $S$ to $\mathbf{R}^{2g} / \Z^{2g}$ coming from the complex-analytic uniformization of $A$ (and it appears in \cite{Boga} and \cite{Law} as well). In particular, in \cite{ACZG}  it is proved that in many cases the Betti map is submersive, which entails density of torsion values in $S(\C)$.  \footnote{The above mentioned paper  \cite{CMZ} also  made use of the Betti map in the unlikely intersection setting as well.}  We also point out Voisin's recent paper \cite{V}, where this density is applied toward problems on Chow groups.

\medskip

The purpose of this note is to show that, when $S$ is a variety over a $p$-adic field, the $p$-adic analytic structure of the torsion locus is much simpler.  Torsion points of different orders stay away from each other, in a sense which will be made precise.\footnote{In the rather special case of elliptic schemes over a curve, a kind of Galois equidistribution of torsion values is proved in \cite{DM-M} both in the complex and $p$-adic case. While the former result implies the complex density, this second conclusion seems not directly related to the present results.}  Our result is similar in form to a theorem of Maulik and Poonen on the jumping locus of the N\'eron--Severi rank \cite{MP}.  

Scanlon proved a related result for subvarieties of a fixed abelian variety: given a subvariety $X$ of an abelian variety, the $p$-adic distances from $X$ to those torsion points not lying on $X$ are bounded below.
We wonder whether a similar result is true for families of abelian varieties.

\begin{question}
\label{scanlon-question}
Let $K$ be a $p$-adic field.  Let $S$ be a quasi-projective variety over $K$, and let $A \rightarrow S$ be an abelian scheme over $S$.
Choose an integral structure on $A$, which gives rise to a notion of $p$-adic distance between points of $A$.
Suppose $X$ is a subscheme of $A$ (resp.\ a subscheme of the form $s(S)$, for $s \colon S \rightarrow A$ a section),
and $Z$ a quasicompact rigid subspace \footnote{A rigid space is \emph{quasicompact} if it is covered by finitely many affinoids; see \cite{BGR} for more on rigid spaces.} of $S$.
Is there a constant $\epsilon$ such that, for any 
torsion section $s \colon Z \rightarrow A$ whose image is disjoint from $X$,
the minimum distance from $s(Z)$ to $X$ is at least $\epsilon$?
\end{question}

\subsection{Rigid analytic spaces.}
To formulate our results, we work in the context of Tate's rigid analytic spaces; for details we refer the reader to \cite{BGR}.  It would be possible, though probably less natural, to avoid this language, and the reader will easily see how to convey the present treatment into a self-contained one.
In any case, we recall here some examples and basic properties of rigid analytic spaces; 

\medskip

Let $K$ be a $p$-adic field, i.e.\ a finite extension of $\mathbf{Q}_p$.

There is an analytification functor allowing one to regard any finite-type $K$-scheme as a rigid analytic space \cite[9.3.4]{BGR}.

The unit ball $\B_n=\B_n(K)=(x_1,\ldots ,x_n)\in K^n: |x_i| \le 1\}$  is a rigid analytic space.
The ring of rigid-analytic functions on $\B_n$ is the \emph{Tate algebra}, the ring of power series in $n$ variables
\[\mathbf{T}_n = K \langle x_1, \ldots, x_n \rangle =  \left \{ \sum a_I x^I, a_I \in K, \lim_{I \rightarrow \infty} a_I = 0 \right \}, \]
where the sum is taken over multi-indices $I$.
This is the same as the ring of power series converging on the unit ball in $ \mathcal{O}_{\mathbf{C}_p}^n$;
see \cite[5.1]{BGR}.

Similarly, for any $r$ in the value group of $\mathbf{C}_p$, one has the notion of rigid-analytic ball of radius $r$.

Given analytic functions $f_1, \ldots, f_k$ on $\B_n$,
their common vanishing locus is again a rigid analytic space;
any function on this space can be expressed (nonuniquely) as an element of
the Tate algebra of functions on $\B_n$.
An analytic space arising in this way is called \emph{affinoid};
affinoids play the same role in the rigid-analytic theory as
affine schemes (spectra of rings) in scheme theory.
See \cite[6.1]{BGR} for more on affinoid rings, \cite[7.1]{BGR} for affinoid varieties,
and \cite[9.1-9.1]{BGR} where affinoids are glued together into general analytic varieties.

A rigid space comes equipped with a \emph{Grothendieck topology}: 
we can speak of \emph{open sets} of such a space, 
and \emph{finite} unions and intersections of open sets are open \cite[9.1]{BGR}.
The restriction to finite unions is a technical device,
meant to overcome the unpleasant fact that a $p$-adic field
with the usual topology is totally disconnected.
It will not concern us here, 
and it will cause the reader no harm to imagine that our rigid spaces come equipped with the \emph{topology} in the usual sense,
a basis for which we now describe.

If $X$ is a rigid space and $f$ an analytic function on $X$, then the subset
\[ \{ x \in X | \left | f(x) \right | \leq 1 \} \]
is an open subspace of $X$; and (if $X$ is affinoid) such opens form  a basis for the (Grothendieck or usual) topology on $X$ \cite[7.2.3/2]{BGR}.
A general rigid variety admits a covering by affinoids.

As a particular example, suppose $X$ is the analytification of a finite-type affine scheme over $K$, and choose coordinates $(x_1, \ldots, x_n)$ on $X$.
Suppose further that $(0, \ldots, 0)$ is a point of $X$.  Then for any $r$, the set 
\[p^r \mathbb{B}_n \cap X = \left \{ (x_1, \ldots, x_n) \in X \text{ such that } |x_i| < |p^r|  \right \}\]
is open in $X$; and conversely any open neighborhood of $(0, \ldots, 0)$ in $X$ contains $p^r \mathbb{B}_n \cap X$ for sufficiently large $r$.

With these preliminaries, we are ready to state our main theorem.

\subsection{The main theorem.}
\begin{thm}
\label{discrete}
Let $K$ be a $p$-adic field.  Let $S$ be a quasi-projective variety over $K$, and let $A \rightarrow S$ be an abelian scheme over $S$.  Choose a section $s: S \rightarrow A$.  For each $n \in \Z_{>0}$, let $S_n \subseteq S$ be the subscheme on which $s$ is torsion of exact order $n$:
\[  S_n = s^{-1}(A[n]) - \bigcup\limits_{d | n,\: d < n} s^{-1}(A[d]).\]
Then for any finite extension $L/K$:
\begin{enumerate}
\item Any $x_0 \in S_n(L)$ has a rigid-analytic neighborhood $U$ that is disjoint from $S_{n'}$ for $n' \neq n$.
\item Suppose $x_0 \in S(L)$ is not in any $S_n$.  Then there is some rigid analytic neighborhood $U$ of $x_0$ which is also disjoint from all of the $S_n$.
\end{enumerate}
\end{thm}

One could ask whether the set of torsion values not only is discrete but has no accumulation points. This is not generally true and we shall give an example in the last section.

\subsection{Acknowledgements}

We would like to thank Yves Andr\'e, Pietro Corvaja, and Shizhang Li for helpful discussions, and the anonymous referee for comments and corrections.  Question \ref{scanlon-question} was suggested to us by the referee.

\section{Proof of Theorem \ref{discrete} }

The proof of Theorem \ref{discrete} uses the following ``structure theorem'' for abelian schemes over a $p$-adic field.

\begin{lem} \label{local_unif}
Let $K$ be a $p$-adic field, $S$ a quasi-projective variety over $K$, and $\pi: A \rightarrow S$ an abelian scheme of dimension $g$ over $S$ with identity section $e: S \rightarrow A$.  Then for every $x_0 \in S(K)$ there exist a rigid-analytic neighborhood $U$ of $x_0$ in $S$, a rigid-analytic open set $E \subset \pi^{-1}(U)$ containing the image of the identity section $e: U \rightarrow A$, and a rigid-analytic map $E \rightarrow U \times \B_g$ to the $p$-adic unit ball over $U$, with the following properties.
\begin{enumerate}
\item \label{isom} The map $E \rightarrow U \times \B_g$ is a rigid-analytic isomorphism.
\item \label{sbgrp} Fiber by fiber, $E$ is a subgroup of $A$: for every finite extension $L$ of $K$, and any $x \in S(L)$, the set $(E \cap A_x)(L)$ is a subgroup of $A_x(L)$.
\item \label{disk} The projection map $E \rightarrow U \times \B_g \rightarrow \B_g$ is a group homomorphism fiber-by-fiber: for any finite extension $L$ of $K$, and any $x \in S(L)$, the map
\( (E \cap A_x) (L) \rightarrow \B_g(L) \)
 is a group homomorphism.  (Here we equip $\B_g(L) = \mathcal{O}_L^{g}$ with the structure of abelian group under addition.)
\item \label{quotient} The subgroup $(E \cap A_{x_0})(K) \subset A_{x_0}(K)$ is of finite index.
\end{enumerate}
\end{lem}

\begin{proof}[Proof of Theorem \ref{discrete}, assuming Lemma \ref{local_unif}.]
Enlarging $K$ if necessary, we may assume that $x_0 \in S(K)$.

Given $x_0 \in S(K)$, take $U_0$ and $E$ as in Lemma \ref{local_unif}.  
Let $n$ be the order of the image of $s(x_0)$ in the finite group $A_{x_0}(K)/((E \cap A_{x_0})(K))$.  Then we have that $n s(x_0) \in E$, so by shrinking $U_0$ if necessary we can assume that $U_0 \subset (ns)^{-1}(E)$.

For any finite extension $L/K$, the group $E(L) \cong \mathcal{O}_L^g$ is torsion-free. Therefore, $s$ is torsion over a point $x \in U_0(L)$ if and only if $ns(x) = 0$; and in this case $s(x)$ is torsion of order dividing $n$.  For $d < n$ dividing $n$, the torsion locus $s^{-1}(A[d])$ is Zariski-closed in $S$ and does not contain $x_0$; so we may assume that $U_0$ is disjoint from $s^{-1}(A[d])$ for such $d$.  Therefore, if $x \in U_0(L)$ and $ns(x) = 0$, then $s(x)$ is torsion of \emph{exact order} $n$.

Now we prove the theorem.  If $t \in S_m(K)$ for some $m$, then we must have $m = n$, and we can take $U = U_0$: no point of $U$ can be torsion of exact order other than $n$.  If $t$ is not torsion, then we can take $U = U_0 - S_n$, so no point of $U$ belongs to the torsion locus.
\end{proof}

\begin{proof}[Proof of Lemma \ref{local_unif}.]

Properties (\ref{isom}) and (\ref{disk}) are a consequence of smoothness, and hold for all group schemes.  Property (\ref{quotient}) is a consequence of compactness.

We prove (\ref{isom}), (\ref{sbgrp}), and (\ref{disk}) first.  Since $A$ is a commutative group scheme over a base of characteristic zero, it is smooth over $S$, and the sheaf of relative differentials $\Omega_{A/S}$ is the pullback of a locally free sheaf of rank $g$ on $S$.  In concrete terms, after restricting to a Zariski neighborhood of $x_0$ in $S$, there are global differentials $\omega_1, \omega_2, \ldots, \omega_g$ on $A$ that form a basis for $\Omega_{A/S}$ at every point of $S$.  These differentials are translation-invariant.

After passing to a rigid open subset of $S$, we can choose local parameters $t_1, \ldots, t_g$ on $A$ at the identity.
In other words, we take $U \subset S$ open, and we require that $t_1, \ldots, t_g$ be meromorphic functions on $A$,
all vanishing at the identity, whose differentials span the cotangent space to $A$ at the identity.
It follows that $t_1, \ldots, t_g$ define an analytic isomorphism $\tau$ from a rigid open subset $E'$ of $A$ to $U \times \mathbf{B}'$,
for $\mathbf{B}'$ a rigid-analytic ball, not necessarily of radius 1.

Since each $\omega_i$ is a global differential, we can write
$$\omega_i = \sum f_{ij}(t_1, \ldots, t_g) d t_j,$$
where each $f$ is a power series in the variables $t_i$ converging on $\mathbf{B}'$.  
Changing coordinates if necessary, we may assume that
$f_{ij}$ takes the value $\delta_{ij}$ (Kronecker delta) at the identity.
Since $d \omega_i = 0$, we can integrate each $\omega_i$ as a formal power series to find a function
$$u_i = u_i(t_1, \ldots, t_g)$$
such that $\omega_i = du_i$, converging on a smaller ball $\mathbf{B} \subseteq \mathbf{B}'$.
By our choice of coordinates, we have $u_i = t_i + \text{(higher order terms)}$.

By the inverse function theorem,
we may assume (perhaps restricting to a still smaller ball) that the functions $u_1, \ldots, u_g$
give an analytic isomorphism $\upsilon$ between $S \times \mathbf{B}$ and itself.
We will take $E = \tau^{-1}(U \times \mathbf{B})$ and then, after rescaling the ball $\mathbf{B}$, the composition $\upsilon \circ \tau \colon E \rightarrow U \times \B_g \subseteq S \times \B_g$ will be our desired map.  Now (\ref{isom}) follows because $\tau$ and $\upsilon$ are both isomorphisms.

For any finite $L / K$ and any point $g \in E(L)$, translation by $g$ leaves the differentials $\omega_i$ invariant; so this translation acts on $u_i$ by addition of a constant.  Similarly, we see that inversion takes $u_i$ to $-u_i$.  
It follows that the group law on $A$ fixes $E$ (i.e.\ $E$ is stable under addition and inversion), and in terms of the coordinates $u_i$ on $\mathbf{B}$, the group law is given by addition.  This proves (\ref{sbgrp}) and (\ref{disk}).

We prove (\ref{quotient}) by a standard compactness argument.  The fiber $A_{x_0}$ is an abelian variety over $K$.
Topologize the set $A_{x_0}(K)$ by the standard $p$-adic topology.  Then $A_{x_0}(K)$ is compact, and cosets
of $(E \cap A_{x_0})(K)$ are open.  Therefore, $(E \cap A_{x_0})(K)$ is of finite index in $A_{x_0}(K)$.
\end{proof}

\section{A counterexample}

If $A$ has bad reduction at a point of $S$, 
Theorem \ref{discrete} no longer holds: torsion points can accumulate above a point of bad reduction.  We consider the following example over $\mathbf{Q}_p$, with $p \neq 2, 3$.  Let $S$ be the disk $p^2 \B_1$, with coordinate $t$; its points over $\mathbf{Q}_p$ are given by $S(\mathbf{Q}_p) = p^2 \Z_p$. Consider the elliptic curve defined over $S$ by the equation
$$E_t: y^2 = (x - 1/12)^2 (x + 1/6) + t \left ( x - \frac{p}{(1-p)^2} - 1/12  \right ).$$
This defines an elliptic scheme $E$.
The fiber over $t=0$ is the nodal cubic
$$E_0: y^2 = (x - 1/12)^2 (x + 1/6).$$
The $j$-invariant of $E_t$ is a rational function of $t$.  After a computation we find that $1/j$ is given by a power series in $\Z_p[[t]]$ such that
$$1 / j(E_t) = - \frac{p}{(1-p)^2} t + O(t^2).$$
In particular, this power series converges for all $t \in S$, and $E_t$ is smooth for all nonzero $t \in S$.

The point
$$(x, y) = \left (\frac{p}{(1-p)^2} + 1/12 , \frac{p(1+p)}{2(1-p)^3} \right )$$
is a point of $E_t$ for every $t$, so it defines a section $s: S \rightarrow E$.  We will prove the following.

\begin{prop}
There exists a sequence $(t_n)_{n>n_0}$ with $t_n \in S$ tending to zero, such that $s_{t_n}$ is torsion in $E_{t_n}$, of exact order $n$.
\end{prop}

We argue using the Tate uniformization.  Recall that the Tate uniformization is defined as follows.  We will work over the base $p \B_1$.  For every $q \in p \B_1(\mathbf{Q}_p) = p \Z_p$, let $\eta_q$ be the map from $\G_m$ to $\P_2$ given in affine coordinates by
$$X(q, z) = \sum_{n \in \Z} \frac{q^n z}{(1 - q^n z)^2} - 2 \sum_{n \geq 1} \frac{n q^n}{1 - q^n} + 1/12 $$
$$Y(q, z) = \sum_{n \in \Z} \frac{q^n z (1 + q^n z)} {2 (1 - q^n z)^3}.$$
(This differs from the formula in \cite[V.3]{Sil2} by the affine-linear transformation $(X, Y) \mapsto (X + 1/12, Y + X/2)$. We have eliminated the $xy$ and $x^2$ terms in the Weierstrass equation at the cost of inverting 2 and 3.)  

For $q \neq 0$, the image of $\eta_q$ is the Tate elliptic curve
$$T_q: y^2 = x^3 + a_4(q)x + a_6(q),$$
and $\eta_q$ induces an isomorphism (in the rigid-analytic sense) between $\mathbf{G}_m / \langle q \rangle$ and $T_q$.  Here $a_4$ and $a_6$ are $p$-adic analytic functions belonging to $\Z_p[[q]]$, with leading terms
$$a_4(q) = -1/48 + O(q)$$
$$a_6(q) = 1/864 + O(q).$$
At $q = 0$, the map $\eta_0$ is given by
$$X(0, z) = \frac{z}{(1-z)^2} + 1/12$$
$$Y(0, z) = \frac{z (1+z)}{2 (1-z)^3},$$ 
and the image of $\eta_0$ is the nodal cubic
$$T_0: y^2 = (x - 1/12)^2 (x + 1/6).$$

The map $\eta_0$ induces an isomorphism between $\mathbf{G}_m$ and the smooth points of $T_0$.  For $q \neq 0$, the uniformization map $\eta_q$ is a group homomorphism from $\mathbf{G}_m$ to $T_q$, with kernel generated by $q$.  In particular, a point $z \in \mathbf{G}_m$ is torsion in the Tate curve if and only if it is of the form $\zeta q^r$, with $\zeta$ a root of unity and $r \in \mathbf{Q}$.
We may regard $\eta$ as an analytic map from $\mathbf{G}_m$ to the Tate curve $T$ over the base $p\B_1$ 

We will use the following explicit form of the Implicit Function Theorem to give a precise disk on which $\eta$ is invertible.

\begin{lem}
\label{implicit1}
Suppose $f \in \Z_p[[ x_1, x_2 ]]$ is such that $f(0, 0) = 0$, and the $x_2$-coefficient of $f$ is a unit in $\Z_p$.
Then we can invert $f$ in the second variable: there is a power series $g\in \Z_p[[ f, x_1 ]]$,
such that $g( f(x_1, x_2), x_1 ) = x_2$ identically as elements of $\Z_p[[x_1, x_2]]$.
In particular, the power series giving $g$ converges for $(f, x_1)$ in any ball of radius less than 1.
\end{lem}

\begin{proof}
One can compute $g=f^{-1} $ (inverse by composition) by successive approximation.
\end{proof}

\begin{lem}
\label{implicit2}
Suppose $f \in \Z_p[[x_1, x_2]]$ is such that $f(a_1, a_2) = 0$, with $a_1, a_2 \in p\Z_p$, and
let $r = \frac{\partial f}{\partial x_2}(a_1, a_2)$.
Suppose $r \neq 0$.
Then we can invert $f$ in the second variable in a neighborhood of $(a_1, a_2)$,
and the inverse $g$  satisfies
\[   g \in r \Z_p [[ \frac{ f }{ r^2 }, \frac{x_1 - a_1}{r^2}.  ]] \]
\end{lem}

\begin{proof}
Apply Lemma \ref{implicit1} to $\frac{1}{r^2} f(a_1 + r^2 x_1, a_2 + r x_2)$; indeed, this lies in $\Z_p[[x_1,x_2]]$ and the $x_2$-coefficient is $1$.
\end{proof}

We want to invert $X(q, z)$ in the second variable in a neighborhood of $(0, p)$.
Writing
$$X(q, z) = \sum_{n \geq 0} \frac{q^n z}{(1 - q^n z)^2} + \sum_{n > 0} \frac{q^n z^{-1}}{(1 - q^n z^{-1})^2} - 2 \sum_{n \geq 1} \frac{n q^n}{1 - q^n} + 1/12, $$
we see that $X$ can be expressed as a Laurent series lying in $\Z_p[[q, z, q z ^{-1}]]$.
Using the identity
\[ \frac{q}{z} = \frac{q}{p} \left( \frac{1}{1 + \frac{z-p}{p}} \right),  \]
we find that
\[ X \in \Z_p[[\frac{q}{p}, \frac{z-p}{p}]].\]
Now we apply Lemma \ref{implicit2} to $X(q, z) - X(0, p)$, with $(x_1, x_2) = (\frac{q}{p}, \frac{z-p}{p})$ and $r = p$, and obtain the following explicit result.

\begin{prop}
There is a power series $\Psi \in p \Z_p [[ \frac{X - X(0, p)}{p^3}, \frac{q}{p^3} ]]$,
such that 
we have $X(q, z) = X$ when $z = \Psi(X, q)$,
provided $X \in (\frac{p}{(1 - p)^2} + 1/12) + p^4 \Z_p$ and $q \in p^4 \Z_p$. 
\end{prop}

Next, we want to show that our family $E$ is isomorphic (in the analytic sense) to the Tate family $T$.  Over $\mathbf{Q}_p$, an elliptic curve is isomorphic to some $T_q$ (with $q$ nonzero) if and only if it has split multiplicative reduction.  (In particular, this only happens when the $j$-invariant is non-integral.)  Every $E_t$ has split multiplicative reduction.  The difficulty is to show that the two families are isomorphic as families, and in particular that the isomorphism extends across the singular fiber $E_0 \cong T_0$.

The Tate curve has $j$-invariant
$$j(T_q) = 1/q + 744 + O(q)$$
 given by a Laurent series in $q$ with leading term $q^{-1}$ and coefficients in $\Z_p$.  Thus
 $$j_T \colon q \mapsto j(T_q)^{-1}$$
is a $p$-adic analytic bijection from the disk $B_p$ to itself.  Similarly, our curve $E_t$ has
$$1 / j(E_t) = - \frac{p}{(1-p)^2} t + O(t^2),$$
so again
$$j_E \colon t \mapsto j(E_t)^{-1}$$
is an analytic map from $S$ to $B_p$.  Combining these maps, we obtain a bijective map
$$\phi = j_T^{-1} \circ j_E: S \rightarrow B_p$$
$$\phi(t) = - \frac{p}{(1-p)^2} t + O(t^2),$$
again given by a power series with coefficients in $\Z_p$, such that $E_t$ and $T_{\phi(t)}$ have the same $j$-invariant.

Next we will show that $E$ is isomorphic to the pullback $\phi^* T$.

\begin{lem}
\label{ec_isom}
Let $R$ be an integral domain, and suppose $A_1, A_2, B_1, B_2  \in R$ are such that $A_2 = \alpha A_1$ and $B_2 = \beta B_1$, for units $\alpha, \beta \in R^*$.  Suppose further that $\alpha$ and $\beta$ satisfy the following.
\begin{enumerate}
\item $\alpha^3 = \beta^2$
\item $\beta / \alpha$ is a square in $R$.
\end{enumerate}
Then the two curves $C_1: y^2 = x^3 + A_1 x + B_1$ and $C_2: y^2 = x^3 + A_2 x + B_2$ are isomorphic over $R$.
\end{lem}

\begin{proof}

In order that the curves be isomorphic over $R$, it is necessary and sufficient that there exist $\lambda \in R^*$ such that
$$A_2 = \lambda^4 A_1$$
and
$$B_2 = \lambda^6 B_1.$$
In this situation, the isomorphism is given by $(x, y) \mapsto (\lambda^2 x, \lambda^3 y)$.

Taking $\lambda$ such that $\lambda^2 = \beta / \alpha$ proves the Lemma.
\end{proof}

Now take 
$$R = \left \{ \sum_{n=0}^{\infty} a_n t^n | a_n \in p^{-2n} \Z_p, \text{ and } \lim_{n \rightarrow \infty} p^{2n} a_n = 0 \right \}.$$
This is the ring of convergent power series on $S$ with Gauss norm (on $S$) at most 1.  Thanks to our choice of $\phi$, the pullback of the Tate family $\phi^* T$ and our family $E$ have the same $j$-invariant.  We apply Lemma \ref{ec_isom} to $C_1 = \phi^* T$ and $C_2 = E$.

All four of $A_1 = a_4(\phi(t))$, $B_1 = a_6(\phi(t))$, $A_2  = -1/48 + t$ and $B_2 = 1/864 -  t \left ( \frac{p}{(1-p)^2} - 1/12  \right )$ are units in $R$, and we define
$$\alpha = A_2 / A_1 = 1 + O(t)$$
$$\beta = B_2 / B_1 = 1 + O(t).$$
The equality of the $j$-invariants gives $\alpha^3 = \beta^2$.  Finally, since $\beta / \alpha$ is a power series in $R$ with leading coefficient $1$, its square root also lies in $R$, by the binomial theorem.

Hence, we may identify $E$ and $\phi^* T$.

The uniformization of the Tate curve gives a surjective map
$$\mathbf{G}_m \times S \rightarrow E$$
over $S$; for $t \in S$ with $t \neq 0$, the kernel of $\mathbf{G}_m \times \{ t \}  \rightarrow  E_t$ is generated by $\phi(t) \in \mathbf{G}_m$.
By the $p$-adic version of the Inverse Function Theorem, the map $\mathbf{G}_m \times S \rightarrow E$ is locally invertible.  Our section
$$s: S \rightarrow E,$$
defined by
$$s(t) = \left (\frac{p}{(1-p)^2} + 1/12, \frac{p(1+p)}{2(1-p)^3} \right )$$
for all $t$,
can be lifted to a map
$$\hat{s}: S \rightarrow \mathbf{G}_m$$
such that $\hat{s}(0) = p$.  
Explicitly, we have
$$\hat{s}(t) = X^{-1} \left (\lambda(t)^2 \left ( \frac{p}{(1-p)^2} + 1/12 \right ), \phi(t) \right). $$
Here $\lambda(t)$ is as in Lemma \ref{ec_isom}; it is a power series in $t$, with leading coefficient $1$ and all coefficients in $\Z_p$;
and $X^{-1}$ is the power series given in Lemma \ref{implicit2}, written abusively as a function of $X$ and $q$.
This power series converges for $t \in B_{p^4} = p^4 \Z_p$, and has the form
$$\hat{s}(t) = p + O(t).$$
For $n \geq 4$ the equation $\phi(t) = \hat{s}(t)^n$ has a unique solution, which can be found by successive approximation, and satisfies $t \in p^n + p^{n+1} \Z_p$.  When $\phi(t) = \hat{s}(t)^n$, the section $s$ is torsion of exact order $n$.  This completes the proof.

\bibliography{padic_bib}
\bibliographystyle{plain}

\end{document}